\documentclass[11pt]{amsart}

\author[C.~Sanna]{Carlo Sanna}
\thanks{$\dagger\,$C.~Sanna is supported by a postdoctoral fellowship of INdAM and is a member of the INdAM group GNSAGA}
\address{Universit\`a di Genova\\Department of Mathematics\\Genova, Italy}
\email{carlo.sanna.dev@gmail.com}

\keywords{binary recurrence sequence; lowest common multiple; Lehmer sequence; random sequence}
\subjclass[2010]{Primary: 11B37, Secondary: 11N37.}
\title[On the l.c.m. of random terms of binary recurrence sequences]{On the l.c.m. of random terms of\\binary recurrence sequences}

\usepackage{amsmath}
\usepackage{amssymb}
\usepackage{amsthm}
\usepackage{geometry}
\usepackage{color}
\usepackage{hyperref}
\usepackage{enumerate}
\hypersetup{colorlinks=true}

\newtheorem{thm}{Theorem}[section]
\newtheorem{cor}{Corollary}[section]
\newtheorem{lem}[thm]{Lemma}
\theoremstyle{remark}

\newcommand{\lcm}{\operatorname{lcm}}
\newcommand{\Li}{\operatorname{Li}}
\allowdisplaybreaks
\begin{document}

\begin{abstract}
For every positive integer $n$ and every $\delta \in [0,1]$, let $B(n, \delta)$ denote the probabilistic model in which a random set $A \subseteq \{1, \dots, n\}$ is constructed by choosing independently every element of $\{1, \dots, n\}$ with probability $\delta$.
Moreover, let $(u_k)_{k \geq 0}$ be an integer sequence satisfying $u_k = a_1 u_{k - 1} + a_2 u_{k - 2}$, for every integer $k \geq 2$, where $u_0 = 0$, $u_1 \neq 0$, and $a_1, a_2$ are fixed nonzero integers; and let $\alpha$ and $\beta$, with $|\alpha| \geq |\beta|$, be the two roots of the polynomial $X^2 - a_1 X - a_2$.
Also, assume that $\alpha / \beta$ is not a root of unity.

We prove that, as $\delta n / \log n \to +\infty$, for every $A$ in $B(n, \delta)$ we have
\begin{equation*}
\log \lcm (u_a : a \in A) \sim \frac{\delta\Li_2(1 - \delta)}{1 - \delta} \cdot \frac{3\log\!\big|\alpha / \!\sqrt{(a_1^2, a_2)}\big|}{\pi^2} \cdot n^2
\end{equation*}
with probability $1 - o(1)$, where $\lcm$ denotes the lowest common multiple, $\Li_2$ is the dilogarithm, and the factor involving $\delta$ is meant to be equal to $1$ when $\delta = 1$.

This extends previous results of Akiyama, Tropak, Matiyasevich, Guy, Kiss and M\'aty\'as, who studied the deterministic case $\delta = 1$, and is motivated by an asymptotic formula for $\lcm(A)$ due to Cilleruelo, Ru\'{e}, \v{S}arka, and Zumalac\'{a}rregui.
\end{abstract}

\maketitle

\section{Introduction}

It is well known that the Prime Number Theorem is equivalent to the asymptotic formula
\begin{equation}\label{equ:lcm123}
\log \lcm(1,2,\dots,n) \sim n ,
\end{equation}
as $n \to +\infty$, where $\lcm$ denotes the lowest common multiple.

For every positive integer $n$ and every $\delta \in [0,1]$, let $B(n, \delta)$ denote the probabilistic model in which a random set $A \subseteq \{1, \dots, n\}$ is constructed by choosing independently every element of $\{1, \dots, n\}$ with probability $\delta$.
Motivated by~\eqref{equ:lcm123}, Cilleruelo, Ru\'{e}, \v{S}arka, and Zumalac\'{a}rregui~\cite{MR3239153} proved the following result (see also~\cite{AKM19} for a more precise version, and~\cite{MR3649012,MR3640773} for others results of similar flavor).

\begin{thm}\label{thm:Cilleruelo}
Let $A$ be a random set in $B(n, \delta)$.
Then, as $\delta n \to +\infty$, we have
\begin{equation*}
\log \lcm (A) \sim \frac{\delta \log(1/\delta)}{1 - \delta} \cdot n ,
\end{equation*}
with probability $1-o(1)$, where the factor involving $\delta$ is meant to be equal to $1$ for $\delta = 1$.
\end{thm}

Let $(u_k)_{k \geq 0}$ be an integer sequence satisfying $u_k = a_1 u_{k - 1} + a_2 u_{k - 2}$, for every integer $k \geq 2$, where $u_0 = 0$, $u_1 \neq 0$, and $a_1, a_2$ are two fixed nonzero integers.
Moreover, let $\alpha$ and $\beta$, with $|\alpha| \geq |\beta|$, be the two roots of the polynomial $X^2 - a_1 X - a_2$.
We assume that $\alpha / \beta$ is not a root of unity, which is a necessary and sufficient condition to have $u_k \neq 0$ for all integers $k \geq 1$.

Akiyama~\cite{MR1077711} and, independently, Tropak~\cite{MR1114366} proved the following analog of~\eqref{equ:lcm123} for the sequence $(u_k)_{k \geq 1}$.

\begin{thm}\label{thm:Akiyama}
We have
\begin{equation*}
\log \lcm (u_1, u_2, \dots, u_n) \sim \frac{3\log\!\big|\alpha / \!\sqrt{(a_1^2, a_2)}\big|}{\pi^2} \cdot n^2 ,
\end{equation*}
as $n \to +\infty$.
\end{thm}

Special cases of Theorem~\ref{thm:Akiyama} were previously proved by Matiyasevich, Guy~\cite{MR1712797}, Kiss and M\'{a}ty\'{a}s~\cite{MR993902}.
Furthermore, Akiyama~\cite{MR1242715, MR1394375} generalized Theorem~\ref{thm:Akiyama} to sequences having some special divisibility properties, while Akiyama and Luca~\cite{MR3150887} studied $\lcm(u_{f(1)}, \dots, u_{f(n)})$ when $f$ is a polynomial, $f = \varphi$ (the Euler's totient function), $f = \sigma$ (the sum of divisors function), or $f$ is a binary recurrence sequence.

Motivated by Theorem~\ref{thm:Cilleruelo}, we give the following generalization of Theorem~\ref{thm:Akiyama}.

\begin{thm}\label{thm:Main}
Let $A$ be a random set in $B(n, \delta)$.
Then, as $\delta n / \log n \to +\infty$, we have
\begin{equation}\label{equ:Main}
\lcm(u_a : a \in A) \sim \frac{\delta\Li_2(1 - \delta)}{1 - \delta} \cdot \frac{3\log\!\big|\alpha / \!\sqrt{(a_1^2, a_2)}\big|}{\pi^2} \cdot n^2 ,
\end{equation}
with probability $1 - o(1)$, where $\Li_2(z) := \sum_{k=1}^\infty z^k / k^2$ is the dilogarithm and the factor involving $\delta$ is meant to be equal to $1$ when $\delta = 1$.
\end{thm}

When $\delta = 1/2$ all the subsets $A \subseteq \{1, \dots, n\}$ are chosen by $B(n, \delta)$ with the same probability.
Hence, Theorem~\ref{thm:Main} together with the identity $\Li_2(\tfrac1{2}) = (\pi^2 - 6(\log 2)^2) / 12$ (see, e.g.,~\cite{MR2290758}) give the following result.

\begin{cor}
As $n \to +\infty$, we have
\begin{equation*}
\lcm(u_a : a \in A) \sim \frac1{4}\left(1 - \frac{6(\log 2)^2}{\pi^2}\right) \cdot \log\!\left|\frac{\alpha}{\sqrt{(a_1^2, a_2)}}\right| \cdot n^2 ,
\end{equation*}
uniformly for all sets $A \subseteq \{1, \dots, n\}$, but at most $o(2^n)$ exceptions.
\end{cor}

\section{Notation}

We employ the Landau--Bachmann ``Big Oh'' and ``little oh'' notations $O$ and $o$, as well as the associated Vinogradov symbols $\ll$ and $\gg$, with their usual meanings.
Any dependence of the implied constants is explicitly stated or indicated with subscripts.
For real random variables $X$ and $Y$, we say that ``$X \sim Y$ with probability $1 - o(1)$'' if $\mathbb{P}\big(|X - Y| \geq \varepsilon|Y|\big) = o_\varepsilon(1)$ for every $\varepsilon > 0$.
We write $\lcm(S)$ for the lowest common multiple of the elements of $S \subseteq \mathbb{Z}$, with the convention $\lcm(\varnothing) := 1$.
We also let $[a,b]$ and $(a,b)$ denote the lowest common multiple and the greatest common divisor, respectively, of two integers $a$ and $b$.
Throughout, the letters $p$ is reserved for prime numbers, and $\nu_p$ denotes the $p$-adic valuation.
As usual, we write $\Lambda(n)$, $\varphi(n)$, $\tau(n)$, and $\mu(n)$, for the von~Mangoldt function, the Euler's totient function, the number of divisors, and the M\"obius function of a positive integer $n$, respectively.

\section{Preliminaries on Lehmer sequences}

Let $\zeta$ and $\eta$ be complex numbers such that $c_1 := (\zeta + \eta)^2$ and $c_2 := \zeta \eta$ are nonzero coprime integers and $\zeta / \eta$ is not a root of unity.
Also, assume $|\zeta| \geq |\eta|$.
The~\emph{Lehmer~sequence} $(\widetilde{u}_k)_{k \geq 0}$ associated to $\zeta$ and $\eta$ is defined~by
\begin{equation}\label{equ:Lehmer}
\widetilde{u}_k := \begin{cases} (\zeta^k - \eta^k) / (\zeta - \eta) & \text{ if $k$ is odd}, \\ (\zeta^k - \eta^k) / (\zeta^2 - \eta^2) & \text{ if $k$ is even},\end{cases}
\end{equation}
for every integer $k \geq 0$.
It is known that $(\widetilde{u}_k)_{k \geq 1}$ is an integer sequence.
For every positive integer $m$ coprime with $c_2$, let $\varrho(m)$ be the \emph{rank of appearance} of $m$ in the Lehmer sequence $(\widetilde{u}_k)_{k \geq 0}$, that is, the smallest positive integer $k$ such that $m \mid \widetilde{u}_k$.
It is known that $\varrho(m)$ exists.
Moreover, for every prime number $p$ not dividing $c_2$, put $\kappa(p) := \nu_p(\widetilde{u}_{\varrho(p)})$.

We need the following properties of the rank of appearance.

\begin{lem}\label{lem:Rank}
We have:
\begin{enumerate}[(i)]
\setlength\itemsep{0.5em}
\item $m \mid \widetilde{u}_k$ if and only if $(m, c_2) = 1$ and $\varrho(m) \mid k$, for all integers $m, k \geq 1$.
\item $\varrho(p^k) = p^{\max(k - \kappa(p))} \varrho(p)$, for all primes $p$ not dividing $2c_2$ and all integers $k \geq 1$.
\item $\varrho(2^k) = 2^{\max(k - \nu_2(\widetilde{u}_{\varrho(4)}))} \varrho(4)$, for all integers $k \geq 2$.
\end{enumerate}
\end{lem}
\begin{proof}
(i) We have $(\widetilde{u}_k, c_2) = 1$ for all integers $k \geq 1$~\cite[Lemma~1]{MR0491445}.
Also, $(\widetilde{u}_k, \widetilde{u}_h) = \widetilde{u}_{(k,h)}$ for all integers $k,h \geq 1$~\cite[Lemma~3]{MR0491445}.
Hence, on the one hand, if $m \mid \widetilde{u}_k$ then $(m, c_2) = 1$ and $m \mid (\widetilde{u}_k, \widetilde{u}_{\varrho(m)}) = \widetilde{u}_{(k, \varrho(m))}$, which in turn implies that $\varrho(m) \mid k$, by the minimality of $\varrho(m)$.
On~the other hand, if $(c_2, m) = 1$ and $\varrho(m) \mid k$ then $m \mid \widetilde{u}_{\varrho(m)} = \widetilde{u}_{(k, \varrho(m))} = (\widetilde{u}_k, \widetilde{u}_{\varrho(m)})$, so that $m \mid \widetilde{u}_k$.

(ii) If $p \mid \tilde{u}_m$, for some positive integer $m$, then $p \mid\mid \widetilde{u}_{pm} / \widetilde{u}_m$~\cite[Lemma~5]{MR0491445}.
Hence, it follows by induction on $h$ that $\nu_p(\widetilde{u}_{p^h \varrho(p)}) = \kappa(p) + h$, for every integer $h \geq 0$.
At this point, the claim follows easily from (i).

(iii) If $4 \mid \tilde{u}_m$, for some positive integer $m$, then $2 \mid\mid \widetilde{u}_{pm} / \widetilde{u}_m$~\cite[Lemma~5]{MR0491445}.
The proof proceeds similarly to the previous point.
\end{proof}

Hereafter, in~light of Lemma~\ref{lem:Rank}(i), in subscripts of sums and products the argument of $\varrho$ is always tacitly assumed to be coprime with $c_2$.

Let us define the cyclotomic numbers $(\phi_k)_{k \geq 1}$ associated to $\zeta$ and $\eta$ by
\begin{equation}\label{equ:phik}
\phi_k := \prod_{\substack{1 \,\leq\, h \,\leq\, k \\ (h,k) \,=\, 1}} \left(\zeta - \mathrm{e}^{\frac{2\pi\mathbf{i}h}{k}} \eta\right) ,
\end{equation}
for every integer $k \geq 0$.
It can be proved that $\phi_k \in \mathbb{Z}$ for every integer $k \geq 3$.
Moreover, from~\eqref{equ:phik} it follows easily that
\begin{equation*}
\zeta^k - \eta^k = \prod_{d \,\mid\, k} \phi_d ,
\end{equation*}
which in turn, applying M\"obius inversion formula and taking into account~\eqref{equ:Lehmer}, gives
\begin{equation}\label{equ:Inversion}
\phi_k = \prod_{d \,\mid\, k} \left(\zeta^d - \eta^d\right)^{\mu(k / d)} = \prod_{d \,\mid\, k} \widetilde{u}_d^{\,\mu(k / d)} ,
\end{equation}
for all integers $k \geq 3$.
We need the following result about $\phi_k$.

\begin{lem}\label{lem:Primitive}
For every integer $k \geq 13$, we have
\begin{equation*}
|\phi_k| = \lambda_k \cdot \prod_{\varrho(p) \,=\, k} p^{\kappa(p)} ,
\end{equation*}
where $\lambda_k$ is equal to $1$ or to the greatest prime factor of $k / (k, 3)$.
\end{lem}
\begin{proof}
Let $p$ be a prime number not dividing $c_2$.
By the definition of $\varrho(p)$, we have that $p \nmid \widetilde{u}_h$ for each positive integer $h < \varrho(p)$.
Hence, by~\eqref{equ:Inversion}, we obtain that $\nu_p(\phi_{\varrho(p)}) = \nu_p(\widetilde{u}_{\varrho(p)}) = \kappa(p)$.
In particular, $p \mid \phi_{\varrho(p)}$.
Let $k \geq 3$ be an integer and suppose that $p$ is a prime factor of $\phi_k$.
On the one hand, if $\varrho(p) = k$ then, by the previous consideration, $\nu_p(\phi_k) = \kappa(p)$.
On the other hand, if $\varrho(p) \neq k$ then $p \mid (\phi_{\varrho(p)}, \phi_k)$.
Finally, for $k \geq 13$ and for every integer $h \geq 3$ with $h \neq k$, we have that $(\phi_h, \phi_k)$ divides the greatest prime factor of $k / (k,3)$~\cite[Lemma~7]{MR0491445}.
\end{proof}

We conclude this section with a formula for a sum involving the von~Mangoldt function.

\begin{lem}\label{lem:Lambda}
We have
\begin{equation}\label{equ:Lambdasum}
\sum_{\varrho(m) \,=\, r} \Lambda(m) = \varphi(r) \log|\zeta| + O_{\zeta, \eta}\!\left(\tau(r)\log (r+1)\right) ,
\end{equation}
and, in particular,
\begin{equation}\label{equ:Lambdabound}
\sum_{\varrho(m) \,=\, r} \Lambda(m) \ll_{\zeta, \eta} \varphi(r) ,
\end{equation}
for every positive integer $r$.
\end{lem}
\begin{proof}
Clearly, we can assume $r \geq 13$.
Write $m = p^k$, where $p$ is a prime number not dividing $c_2$ and $k$ is a positive integer.
First, suppose that $p > 2$.
By Lemma~\ref{lem:Rank}(ii), we have that $\varrho(m) = p^{\max(k - \kappa(p), 0)} \varrho(p)$.
Hence, $\varrho(m) = r$ if and only if $k \leq \kappa(p)$ and $\varrho(p) = r$, or $k > \kappa(p)$ and $p^{k - \kappa(p)} \varrho(p) = r$.
In the first case, the contribution to the sum in~\eqref{equ:Lambdasum} is exactly $\kappa(p) \log p$.
In the second case, $p \mid r$ and, since $k$ is determined by $p$ and $r$, the contribution to the sum in~\eqref{equ:Lambdasum} is $\log p$.
Using Lemma~\ref{lem:Rank}(iii), the case $p = 2$ can be handled similarly.
Therefore,
\begin{equation}\label{equ:Lambda1}
\sum_{\varrho(m) \,=\, r} \Lambda(m) = \sum_{\varrho(p) \,=\, r} \kappa(p) \log p + O\!\left(\sum_{p \,\mid\, r} \log p\right) = \log |\phi_r| + O(\log r) ,
\end{equation}
where we used Lemma~\ref{lem:Primitive}.
Furthermore, from~\eqref{equ:Inversion} and the the identity $\sum_{d \,\mid\, r} \mu(r / d)\,d = \varphi(r)$, it follows that
\begin{equation*}
\log|\phi_r| = \varphi(r) \log |\zeta| + O\!\left(\sum_{d \,\mid\, r} \log\!\Big|1 - \big(\tfrac{\eta}{\zeta}\big)^d\Big|\right) .
\end{equation*}
If $|\eta / \zeta| < 1$ then $\log\!\big|1 - (\eta / \zeta)^d\big| = O_{\zeta,\eta}(1)$.
If $|\eta / \zeta| = 1$ then, since $\eta / \zeta$ is an algebraic number that is not a root of unity, it follows from classic bounds on linear forms in logarithms (see, e.g.,~\cite[Lemma~3]{MR951903}) that $\log\!\big|1 - (\eta / \zeta)^d\big| = O_{\zeta,\eta}(\log(d + 1))$.
Consequently,
\begin{equation}\label{equ:Lambda2}
\log|\phi_r| = \varphi(r) \log |\zeta| + O_{\zeta, \eta}\!\left(\tau(r) \log (r + 1) \right) .
\end{equation}
Putting together~\eqref{equ:Lambda1} and~\eqref{equ:Lambda2}, we get~\eqref{equ:Lambdasum}.
Finally, the upper bound~\eqref{equ:Lambdabound} follows since $\tau(k) \leq k^{\varepsilon}$ and $\varphi(k) \geq k^{1 - \varepsilon}$, for all $\varepsilon > 0$ and every integer $k \gg_\varepsilon 1$~\cite[Ch.~I.5, Corollary~1.1 and Eq.~12]{Ten95}.
\end{proof}

\section{Further preliminaries}

We need two estimates involving the Euler's totient function.
Define
\begin{equation*}
\Phi(x) := \sum_{n \,\leq\, x} \varphi(n) ,
\end{equation*}
for every $x \geq 1$.

\begin{lem}\label{lem:Phi}
We have
\begin{equation*}
\Phi(x) = \frac{3}{\pi^2} \, x^2 + O(x \log x) \quad\text{and}\quad \sum_{n \,\leq\, x} \frac{\varphi(n)}{n} \ll x ,
\end{equation*}
for every $x \geq 2$.
\end{lem}
\begin{proof}
The first formula is well known~\cite[Ch.~I.3, Thm.~4]{Ten95} and implies
\begin{equation*}
\sum_{n \,\leq\, x} \frac{\varphi(n)}{n} \leq \sum_{n \,\leq\, x / 2} 1 + \sum_{x/2 \,<\, n \,\leq\, x} \frac{\varphi(n)}{x/2} \ll x ,
\end{equation*}
as desired.
\end{proof}

The following lemma is an easy inequality that will be useful later.

\begin{lem}\label{lem:Bernoulli}
It holds $1 - (1 - x)^k \leq k x$, for all $x \in [0,1]$ and all integers $k \geq 0$.
\end{lem}
\begin{proof}
The claim is $(1 + (-x))^k \geq 1 + k(-x)$, which follows from Bernoulli's inequality.
\end{proof}

\section{Proof of Theorem~\ref{thm:Main}}

Henceforth, all the implied constants may depend by $a_1$, $a_2$, and $u_1$.
It is well known that the generalized Binet's formula
\begin{equation}\label{equ:Binet}
u_k = \frac{\alpha^k - \beta^k}{\alpha - \beta} \, u_1 ,
\end{equation}
holds for every integer $k \geq 0$.
We put $\zeta := \alpha / \!\!\;\sqrt{b}$ and $\eta := \beta / \!\!\;\sqrt{b}$, where $b := (a_1^2, a_2)$.
Note that indeed $c_1 = a_1^2 / b$ and $c_2 = -a_2 / b$ are nonzero relatively prime integers, $\zeta / \eta = \alpha / \beta$ is not a root of unity, and $|\zeta| \geq |\eta|$.
Moreover, from~\eqref{equ:Lehmer} and~\eqref{equ:Binet}, it follows easily that
\begin{equation*}
u_k = \begin{cases} b^{(k - 1) / 2}u_1 \widetilde{u}_k & \text{ if $k$ is odd}, \\ a_1 b^{k / 2 - 1}u_1 \widetilde{u}_k & \text{ if $k$ is even}, \end{cases}
\end{equation*}
for every integer $k \geq 0$.
Therefore, for every $A \subseteq \{1,\dots,n\}$, we have
\begin{equation*}
\log\lcm(u_a : a \in A) = \log\lcm(\widetilde{u}_a : a \in A) + O(n) .
\end{equation*}
Note that $O(n)$ is a ``little oh'' of the right-hand side of~\eqref{equ:Main}, as $\delta n / \log n \to +\infty$.
Hence, it is enough to prove Theorem~\ref{thm:Main} with $\log\lcm(\widetilde{u}_a : a \in A)$ in place of $\log\lcm(u_a : a \in A)$, and this will be indeed our strategy.

Hereafter, let $A$ be a random set in $B(n, \delta)$, and put $L := \lcm(\widetilde{u}_a : a \in A)$ and $X := \log L$.
For every positive integer $m$ coprime with $c_2$, let us define
\begin{equation*}
I_A(m) := \begin{cases}1 & \text{ if } \varrho(m) \mid a \text{ for some } a \in A,\\ 0 & \text{ otherwise.}\end{cases}
\end{equation*}
The following lemma gives an expression for $X$ in terms of $I_A$ and the von~Mangoldt function.

\begin{lem}\label{lem:XLIA}
We have
\begin{equation*}
X = \sum_{\varrho(m) \,\leq\, n} \Lambda(m)\,I_A(m) .
\end{equation*}
\end{lem}
\begin{proof}
For every prime power $p^k$ with $p \nmid c_2$, we know from Lemma~\ref{lem:Rank}(i) that $p^k \mid L$ if and only if $\varrho(p^k) \mid a$ for some $a \in A$ and, in~particular, $\varrho(p^k) \leq n$.
Hence,
\begin{equation*}
X = \sum_{p^k \,\mid\, L} \log p = \sum_{\varrho(p^k) \,\leq\, n} (\log p)\,I_A\big(p^k\big) = \sum_{\varrho(m) \,\leq\, n} \Lambda(m)\,I_A(m) ,
\end{equation*}
as claimed.
\end{proof}

The next lemma provides two expected values involving $I_A$ and needed in later arguments.

\begin{lem}\label{lem:EIA}
We have
\begin{equation}\label{equ:EIA1}
\mathbb{E}\big(I_A(m)\big) = 1 - (1 - \delta)^{\lfloor n / \varrho(m)\rfloor}
\end{equation}
and
\begin{align*}
\mathbb{E}\big(I_A(m)I_A(\ell)\big) = 1 - (1 - \delta)^{\lfloor n / \varrho(m)\rfloor} \;- &\;(1 - \delta)^{\lfloor n / \varrho(\ell)\rfloor} \\
&+ (1 - \delta)^{\lfloor n / \varrho(m)\rfloor + \lfloor n / \varrho(\ell)\rfloor - \lfloor n / [\varrho(m),\varrho(\ell)]\rfloor} ,
\end{align*}
for all positive integers $m$ and $\ell$ with $(m\ell, c_2) = 1$.
\end{lem}
\begin{proof}
By the definition of $I_A$, we have
\begin{equation*}
\mathbb{E}\big(I_A(m)\big) = \mathbb{P}\big(\exists a \in A : \varrho(m) \mid a\big) = 1 - \mathbb{P}\!\left(\bigwedge_{t \,\leq\, n / \varrho(m)} (\varrho(m)t \notin A)\right) = 1 - (1 - \delta)^{\lfloor n / \varrho(m)\rfloor} ,
\end{equation*}
which is the first claim.
On the one hand, by linearity of expectation and by~\eqref{equ:EIA1}, we obtain
\begin{align*}
\mathbb{E}\big(I_A(m)I_A(\ell)\big) &= \mathbb{E}\!\left(I_A(m) + I_A(\ell) - 1 + \big(1-I_A(m)\big)\big(1-I_A(\ell)\big)\right) \\
&= \mathbb{E}\big(I_A(m)\big) + \mathbb{E}\big(I_A(\ell)\big) - 1 + \mathbb{E}\!\left(\big(1-I_A(m)\big)\big(1-I_A(\ell)\big)\right) \\
&= 1 - (1 - \delta)^{\lfloor n / \varrho(m)\rfloor} - (1 - \delta)^{\lfloor n / \varrho(\ell)\rfloor} + \mathbb{E}\big((1-I_A(m))(1-I_A(\ell))\big) .
\end{align*}
On the other hand, by the definition of $I_A$,
\begin{align*}
\mathbb{E}&\left(\big(1-I_A(m)\big)\big(1-I_A(\ell)\big)\right) = \mathbb{P}\big(\forall a \in A : \varrho(m) \nmid a \text{ and } \varrho(\ell) \nmid a\big) \\
&= \mathbb{P}\!\left(\bigwedge_{\substack{k \,\leq\, n \\ \varrho(m)\,\mid\,k \text{ or } \varrho(\ell)\,\mid\,k}} (k \notin A)\right) = (1 - \delta)^{\lfloor n / \varrho(m)\rfloor + \lfloor n / \varrho(\ell)\rfloor - \lfloor n / [\varrho(m), \varrho(\ell)]\rfloor} ,
\end{align*}
and the second claim follows too.
\end{proof}

Now we give an asymptotic formula for the expected value of $X$.

\begin{lem}\label{lem:EX}
We have
\begin{equation*}
\mathbb{E}(X) = \frac{\delta \Li_2(1 - \delta)}{1 - \delta} \cdot \frac{3 \log |\zeta|}{\pi^2} \cdot n^2 + O\!\left(\delta n (\log n)^3 \right) ,
\end{equation*}
for all integers $n \geq 2$.
In particular,
\begin{equation*}
\mathbb{E}(X) \sim \frac{\delta \Li_2(1 - \delta)}{1 - \delta} \cdot \frac{3 \log |\zeta|}{\pi^2} \cdot n^2 ,
\end{equation*}
as $n \to +\infty$, uniformly for $\delta \in {(0,1]}$.
\end{lem}
\begin{proof}
From Lemma~\ref{lem:XLIA} and Lemma~\ref{lem:EIA}, it follows that
\begin{align*}
\mathbb{E}(X) &= \sum_{\varrho(m) \,\leq\, n} \Lambda(m)\,\mathbb{E}\big(I_A(m)\big) \\
&= \sum_{\varrho(m) \,\leq\, n} \Lambda(m)\,\big(1 - (1 - \delta)^{\lfloor n / \varrho(m)\rfloor}\big) \\
&= \sum_{r \,\leq\, n} \big(1 - (1 - \delta)^{\lfloor n / r\rfloor}\big) \sum_{\varrho(m) \,=\, r} \Lambda(m) .
\end{align*}
Consequently, thanks to Lemma~\ref{lem:Lambda} and Lemma~\ref{lem:Bernoulli}, we obtain
\begin{align}\label{equ:EX1}
\mathbb{E}(X) &= \sum_{r \,\leq\, n} \big(1 - (1 - \delta)^{\lfloor n / r\rfloor}\big) \,\varphi(r) \log|\zeta| + O\!\left(\delta n \sum_{r \,\leq\, n} \frac{\tau(r)\log (r+1)}{r}\right) \\
&= \sum_{r \,\leq\, n} \big(1 - (1 - \delta)^{\lfloor n / r\rfloor}\big) \,\varphi(r) \log|\zeta| + O\!\left(\delta n (\log n)^3\right) , \nonumber
\end{align}
where we used the fact that
\begin{equation*}
\sum_{r \,\leq\, n} \frac{\tau(r)}{r} \leq \Big(\sum_{s \,\leq\, n} \frac1{s}\Big)^2 \ll (\log n)^2 .
\end{equation*}
Note that $\lfloor n / r \rfloor = j$ if and only if $r \in {(n / (j + 1), n / j]}$.
Hence,
\begin{align}\label{equ:EX2}
\sum_{r \,\leq\, n} &\big(1 - (1 - \delta)^{\lfloor n / r \rfloor}\big) \,\varphi(r) = \sum_{j \,\leq\, n} \big(1 - (1 - \delta)^j \big) \sum_{n / (j + 1)\,<\, r \,\leq\, n / j} \varphi(r) \\
&= \sum_{j \,\leq\, n} \big(1 - (1 - \delta)^j \big) \left(\Phi\!\left(\frac{n}{j}\right) - \Phi\!\left(\frac{n}{j+1}\right)\right) \nonumber\\
&= \delta \sum_{j \,\leq\, n} (1 - \delta)^{j-1} \,\Phi\!\left(\frac{n}{j}\right) \nonumber\\
&= \delta \sum_{j \,\leq\, n} \frac{(1 - \delta)^{j-1}}{j^2} \cdot \frac{3}{\pi^2}\cdot n^2 + O\!\left(\delta \sum_{j \,\leq\, n} \frac{n}{j} \log\!\left(\frac{n}{j}\right)\right) \nonumber\\
&= \frac{\delta \Li_2(1 - \delta)}{1 - \delta} \cdot \frac{3}{\pi^2}\cdot n^2 + O\!\left(\delta n (\log n)^2\right) , \nonumber
\end{align}
where we used Lemma~\ref{lem:Phi}.
Finally, putting together~\eqref{equ:EX1} and \eqref{equ:EX2}, we get the desired claim.
\end{proof}

The next lemma is an upper bound for the variance of $X$.

\begin{lem}\label{lem:VX}
We have
\begin{equation*}
\mathbb{V}(X) \ll \delta n^3 \log n,
\end{equation*}
for all integers $n \geq 2$.
\end{lem}
\begin{proof}
On the one hand, by Lemma~\ref{lem:XLIA}, we have
\begin{align*}
\mathbb{V}(X) &= \mathbb{E}\big(X^2\big) - \mathbb{E}(X)^2 \\
&= \sum_{\varrho(m),\,\varrho(\ell) \,\leq\, n} \Lambda(m)\Lambda(\ell)\!\left(\mathbb{E}\big(I_A(m)I_A(\ell)\big) - \mathbb{E}\big(I_A(m)\big)\mathbb{E}\big(I_A(\ell)\big)\right) .
\end{align*}
On the other hand, from Lemma~\ref{lem:EIA} and Lemma~\ref{lem:Bernoulli}, it follows that
\begin{align*}
\mathbb{E}&\big(I_A(m)I_A(\ell)\big) - \mathbb{E}\big(I_A(m)\big)\mathbb{E}\big(I_A(\ell)\big) \\
&= (1 - \delta)^{\lfloor n / \varrho(m) \rfloor + \lfloor n / \varrho(\ell) \rfloor - \lfloor n / [\varrho(m),\varrho(\ell)] \rfloor} \big(1 - (1 - \delta)^{\lfloor n / [\varrho(m), \varrho(\ell)]\rfloor}\big) \leq \frac{\delta n}{[\varrho(m), \varrho(\ell)]} .
\end{align*}
Therefore,
\begin{align}\label{equ:VX1}
\mathbb{V}(X) &\leq \delta n \sum_{\varrho(m),\,\varrho(\ell) \,\leq\, n} \frac{\Lambda(m)\,\Lambda(\ell)}{[\varrho(m), \varrho(\ell)]} = \delta n \sum_{r,s \,\leq\, n} \frac1{[r,s]} \sum_{\varrho(m) \,=\, r} \Lambda(r) \sum_{\varrho(\ell) \,=\, s} \Lambda(\ell) \\
&\ll \delta n \sum_{r,s \,\leq\, n} \frac{\varphi(r)\,\varphi(s)}{[r,s]} = \delta n \sum_{r,s \,\leq\, n} (r,s)\,\frac{\varphi(r)\,\varphi(s)}{rs} , \nonumber
\end{align}
where we used Lemma~\ref{lem:Lambda} and the identity $[r,s] = rs / (r, s)$.
At this point, writing $r = dr^\prime$ and $s = ds^\prime$, where $d := (r, s)$, we obtain
\begin{align}\label{equ:VX2}
\sum_{r,s \,\leq\, n} (r,s)\,\frac{\varphi(r)\,\varphi(s)}{rs} &= \sum_{d \,\leq\, n} d \sum_{\substack{r^\prime\!, s^\prime \,\leq\, n / d \\ (r^\prime\!,s^\prime) \,=\, 1}} \frac{\varphi(dr^\prime)\,\varphi(ds^\prime)}{d^2 r^\prime s^\prime} \leq \sum_{d \,\leq\, n} d \,\Big(\sum_{t \,\leq\, n / d} \frac{\varphi(t)}{t}\Big)^2 \\
&\ll \sum_{d \,\leq\, n} d \left(\frac{n}{d}\right)^2 \ll n^2 \log n , \nonumber
\end{align}
where we used Lemma~\ref{lem:Phi} and the inequality $\varphi(dm) \leq d\varphi(m)$, holding for every integer $m \geq 1$.
Finally, putting together~\eqref{equ:VX1} and \eqref{equ:VX2}, we get the desired claim.
\end{proof}

\begin{proof}[Proof of Theorem~\ref{thm:Main}]
By Chebyshev's inequality, Lemma~\ref{lem:EX}, and Lemma~\ref{lem:VX}, we have
\begin{equation*}
\mathbb{P}\big(|X - \mathbb{E}(X)| \geq \varepsilon\mathbb{E}(X)\big) \leq \frac{\mathbb{V}(X)}{\big(\varepsilon \mathbb{E}(X)\big)^2} \ll \frac{\log n}{\varepsilon^2 \delta n} = o_\varepsilon(1) ,
\end{equation*}
as $\delta n / \log n \to +\infty$.
Hence, again by Lemma~\ref{lem:EX}, we have
\begin{equation*}
X \sim \frac{\delta \Li_2(1 - \delta)}{1 - \delta} \cdot \frac{3 \log |\zeta|}{\pi^2} \cdot n^2 ,
\end{equation*}
with probability $1 - o(1)$, as desired.
\end{proof}

\bibliographystyle{amsplain}
\providecommand{\bysame}{\leavevmode\hbox to3em{\hrulefill}\thinspace}

\end{document}